\documentclass[11pt,english]{article}
\usepackage[latin9]{inputenc}
\usepackage[a4paper]{geometry}
\usepackage{verbatim}
\usepackage{amsthm}
\usepackage{amsmath}
\usepackage{amssymb}

\makeatletter
\usepackage{enumitem}       
\theoremstyle{plain}
\newtheorem{thm}{\protect\theoremname}
  \theoremstyle{plain}
  
  \theoremstyle{remark}
  \newtheorem{claim}[thm]{\protect\claimname}

\makeatother

\usepackage{babel}
  \providecommand{\claimname}{Claim}
  \providecommand{\lemmaname}{Lemma}
\providecommand{\theoremname}{Theorem}

\begin{document}

\title{A construction of almost Steiner systems\global\long\def\prob{\mathrm{Pr}}
\global\long\def\E{\mathbb{E}}
\global\long\def\Bin{\mathrm{Bin}}
\global\long\def\Po{\mathrm{Po}}
}

\author{Asaf Ferber%
\thanks{School of Mathematical Sciences, Raymond and Beverly Sackler Faculty
of Exact Sciences, Tel Aviv University, Tel Aviv, 69978, Israel. Email:
ferberas@tau.ac.il.%
} \and Rani Hod%
\thanks{School of Computer Science, Raymond and Beverly Sackler Faculty of
Exact Sciences, Tel Aviv University, Tel Aviv, 69978, Israel. Email:
rani.hod@cs.tau.ac.il. Research supported by an ERC advanced grant.%
} \and Michael Krivelevich%
\thanks{School of Mathematical Sciences, Raymond and Beverly Sackler Faculty
of Exact Sciences, Tel Aviv University, Tel Aviv, 69978, Israel. Email:
krivelev@tau.ac.il. Research supported in part by USA-Israel BSF Grant
2010115 and by grant 912/12 from the Israel Science Foundation.%
} \and Benny Sudakov%
\thanks{Department of Mathematics, University of California, Los Angeles,
CA, USA. Email: bsudakov@math.ucla.edu. Research supported in part
by NSF grant DMS-1101185, by AFOSR MURI grant FA9550-10-1-0569 and
by a USA-Israel BSF grant.%
}}
\maketitle
\begin{abstract}
Let $n$, $k$, and $t$ be integers satisfying $n>k>t\ge2$. A Steiner
system with parameters $t$, $k$, and $n$ is a $k$-uniform hypergraph
on $n$ vertices in which every set of $t$ distinct vertices is contained
in exactly one edge. An outstanding problem in Design Theory is to
determine whether a nontrivial Steiner system exists for $t\geq6$.

In this note we prove that for every $k>t\ge2$ and sufficiently large
$n$, there exists an almost Steiner system with parameters $t$,
$k$, and $n$; that is, there exists a $k$-uniform hypergraph on
$n$ vertices such that every set of $t$ distinct vertices is covered
by either one or two edges.
\end{abstract}

\section{Introduction}

Let $n$, $k$, $t$, and $\lambda$ be positive integers satisfying
$n>k>t\geq2$. A \emph{$t$-$\left(n,k,\lambda\right)$-design} is
a $k$-uniform hypergraph $\mathcal{H}=\left(X,\mathcal{F}\right)$
on $n$ vertices with the following property: every $t$-set of vertices
$A\subset X$ is contained in exactly $\lambda$ edges $F\in\mathcal{F}$.
The special case $\lambda=1$ is known as a \emph{Steiner system}
with parameters $t$, $k$, and $n$, named after Jakob Steiner who
pondered the existence of such systems in 1853. Steiner systems, $t$-designs%
\footnote{That is, $t$-$\left(n,k,\lambda\right)$-designs for some parameters
$n$,$k$, and $\lambda$.%
} and other combinatorial designs turn out to be useful in a multitude
of applications, e.g., in coding theory, storage systems design, and
wireless communication. For a survey of the subject, the reader is
referred to~\cite{handbook-of-combinatorial-designs}.

A counting argument shows that a Steiner triple system --- that is,
a $2$-$\left(n,3,1\right)$-design --- can only exist when $n\equiv1\textrm{ or }n\equiv3\pmod6$.
For every such $n$, this is achieved via constructions based on symmetric
idempotent quasigroups. Geometric constructions over finite fields
give rise to some further infinite families of Steiner systems with
$t=2$ and $t=3$. For instance, for a prime power $q$ and an integer
$m\ge2$, affine geometries yield $2$-$\left(q^{m},q,1\right)$-designs,
projective geometries yield $2$-$\left(q^{m}+\cdots+q^{2}+q+1,q+1,1\right)$-designs
and spherical geometries yield $3$-$\left(q^{m}+1,q,1\right)$-designs.

For $t=4$ and $t=5$, only finitely many nontrivial constructions
of Steiner systems are known; for $t\ge6$, no constructions are known
at all.

\medskip{}

Before stating our result, let us extend the definition of $t$-designs
as follows. Let $n$, $k$, and $t$ be positive integers satisfying
$n>k>t\geq2$ and let $\Lambda$ be a set of positive integers. A
$t$-$\left(n,k,\Lambda\right)$-design is a $k$-uniform hypergraph
$\mathcal{H}=\left(X,\mathcal{F}\right)$ on $n$ vertices with the
following property: for every $t$-set of vertices $A\subset X$,
the number of edges $F\in\mathcal{F}$ that contain $A$ belongs to
$\Lambda$. Clearly, when $\Lambda=\left\{ \lambda\right\} $ is a
singleton, a $t$-$\left(n,k,\left\{ \lambda\right\} \right)$-design
coincides with a $t$-$\left(n,k,\lambda\right)$-design as defined
above.

Not able to construct Steiner systems for large $t$, Erd\H{o}s and
Hanani~\cite{EH63} aimed for large partial Steiner systems; that
is, $t$-$\left(n,k,\left\{ 0,1\right\} \right)$-designs with as
many edges as possible. Since a Steiner system has exactly ${n \choose t}/{k \choose t}$
edges, they conjectured the existence of partial Steiner systems with
$\left(1-o\left(1\right)\right){n \choose t}/{k \choose t}$ edges.
This was first proved by Rödl~\cite{Rodl85} in 1985, with further
refinements~\cite{Grable99,Kim01,KR98} of the $o\left(1\right)$
term, as stated in the following theorem:
\begin{thm}[Rödl]
\label{thm:partial-steiner}Let $k$ and $t$ be integers such that
$k>t\ge2$ . Then there exists a partial Steiner system with parameters
$t$, $k$, and $n$ covering all but $o\left(n^{t}\right)$ of the
$t$-sets.
\end{thm}
Theorem~\ref{thm:partial-steiner} can also be rephrased in terms of
a covering rather than a packing; that is, it asserts the existence
of a system with $\left(1+o\left(1\right)\right){n \choose t}/{k
\choose t}$ edges such that every $t$-set is covered at least once
(see, e.g.,~\cite[page 56]{the-probabilistic-method}). Nevertheless,
some $t$-sets might be covered multiple times (perhaps even
$\omega\left(1\right)$ times). It is therefore natural to ask for
$t$-$\left(n,k,\left\{ 1,\ldots,r\right\} \right)$-designs, where
$r$ is as small as possible. The main aim of this short note is to
show how to extend Theorem~\ref{thm:partial-steiner} to cover all
$t$-sets at least once but at most twice.
\begin{thm}
\label{thm:main}Let $k$ and $t$ be integers such that $k>t\ge2$.
Then, for sufficiently large $n$, there exists a $t$-$\left(n,k,\left\{ 1,2\right\} \right)$-design.
\end{thm}

Our proof actually gives a stronger result: there exists a
$t$-$\left(n,k,\left\{ 1,2\right\} \right)$-design with
$\left(1+o\left(1\right)\right){n \choose t}/{k \choose t}$ edges.

\section{Preliminaries}

In this section we present results needed for the proof of Theorem
\ref{thm:main}.\medskip{}

Given a $t$-$\left(n,k,\left\{ 0,1\right\} \right)$-design $\mathcal{H}=\left(X,\mathcal{F}\right)$,
we define the \emph{leave hypergraph} $\left(X,\mathcal{L}_{\mathcal{H}}\right)$
to be the $t$-uniform hypergraph whose edges are the $t$-sets $A\subset X$
not covered by any edge $F\in\mathcal{F}$.

Following closely the proof of Theorem~\ref{thm:partial-steiner}
appearing in~\cite{Grable99}, we recover an extended form of the
theorem, which is a key ingredient in the proof of our main result.
\begin{thm}
\label{thm:partial-steiner-ext}Let $k$ and $t$ be integers such that
$k>t\ge2$. There exists a constant
$\varepsilon=\varepsilon\left(k,t\right)>0$ such that for
sufficiently large $n$, there exists a partial Steiner system
$\mathcal{H}=\left(X,\mathcal{F}\right)$ with parameters $t$, $k$,
and $n$ satisfying the following property:
\begin{itemize}
\item [$\left(\clubsuit\right)$] For every $0\le\ell <t$,  every set $X'\subset X$ of size $\left|X'\right|=\ell$
is contained in $O\left(n^{t-\ell-\varepsilon}\right)$ edges of the
leave hypergraph.
\end{itemize}
\end{thm}

We also make use of the following probabilistic tool.

%
%

\paragraph{Talagrand's inequality.}

In its general form, Talagrand's inequality is an isoperimetric-type
inequality for product probability spaces. We use the following formulation
from~\cite[pages 232--233]{graph-colouring-and-the-probabilistic-method},
suitable for showing that a random variable in a product space is
unlikely to overshoot its expectation under two conditions:
\begin{thm}[Talagrand]
\label{thm:talagrand-ineq}Let $Z\ge0$ be a non-trivial random variable,
which is determined by $n$ independent trials $T_{1},\ldots,T_{n}$.
Let $c>0$ and suppose that the following properties hold:
\begin{enumerate}[label=\roman*.]
\item ($c$-Lipschitz) changing the outcome of one trial can affect $Z$ by
at most $c$, and
\item (Certifiable) for any $s$, if $Z\ge s$ then there is a set of at
most $s$ trials whose outcomes certify that $Z\ge s$.
\end{enumerate}
Then $\prob\left[Z>t\right]<2\exp\left(-t/16c^{2}\right)$ for any
$t\ge2\E\left[Z\right]+80c\sqrt{\E\left[Z\right]}$.
\end{thm}

\section{Proof of the main result}

In this section we prove Theorem~\ref{thm:main}.

\subsection{Outline}

The construction is done in two phases:
\begin{enumerate}[label=\Roman*.]
\item \label{enu:phase-1}Apply Theorem~\ref{thm:partial-steiner-ext}
to get a $t$-$\left(n,k,\left\{ 0,1\right\} \right)$-design
$\mathcal{H}=\left(X,\mathcal{F}\right)$ with
property~$\left(\clubsuit\right)$ with respect to some
$0<\varepsilon<1$.
\item \label{enu:phase-2}Build another $t$-$\left(n,k,\left\{ 0,1\right\} \right)$-design
$\mathcal{H}'=\left(X,\mathcal{F}'\right)$ that covers the uncovered
$t$-sets $\mathcal{L}_{\mathcal{H}}$.
\end{enumerate}
Combining both designs, we get that every $t$-set is covered at least
once but no more than twice; namely $\left(X,\mathcal{F}\cup\mathcal{F}'\right)$
is a $t$-$\left(n,k,\left\{ 1,2\right\} \right)$-design, as required.

\medskip{}

We now describe how to build $\mathcal{H}'$. For a set $A\subset X$,
denote by $\mathcal{T}_{A}=\left\{ C\subseteq X:\left|C\right|=k\mbox{ and }A\subseteq C\right\} $
the family of possible continuations of $A$ to a subset of $X$ of
cardinality $k$. Note that $\mathcal{T}_{A}=\varnothing$ when $\left|A\right|>k$.

Consider the leave hypergraph $\left(X,\mathcal{L}_{\mathcal{H}}\right)$.
Our goal is to choose, for every uncovered $t$-set $A\in\mathcal{L}_{\mathcal{H}}$,
a $k$-set $A'\in\mathcal{T}_{A}$ such that $\left|A'\cap B'\right|<t$
for every two distinct $A,B\in\mathcal{L}_{\mathcal{H}}$. This ensures
that the obtained hypergraph $\mathcal{H}'=\left(X,\left\{ A':A\in\mathcal{L}_{\mathcal{H}}\right\} \right)$
is indeed a $t$-$\left(n,k,\left\{ 0,1\right\} \right)$-design.

To this aim, for every $A\in\mathcal{L}_{\mathcal{H}}$ we introduce
intermediate lists $\mathcal{R}_{A}\subseteq\mathcal{S}_{A}\subseteq\mathcal{T}_{A}$
that will help us control the cardinalities of pairwise intersections
when choosing $A'\in\mathcal{R}_{A}$. First note that we surely cannot
afford to consider continuations that fully contain some other $B\in\mathcal{L}_{\mathcal{H}}$,
so we restrict ourselves to the list
\[
\mathcal{S}_{A}=\mathcal{T}_{A}\setminus\bigcup\left\{ \mathcal{T}_{B}:B\in\mathcal{L}_{H}\setminus\left\{ A\right\} \right\} =\left\{ C\in\mathcal{T}_{A}:B\not\subseteq C\mbox{ for all }B\in\mathcal{L}_{H},B\ne A\right\} .
\]
Note that, by definition, the lists $\mathcal{S}_{A}$ for different
$A$ are disjoint. Next, choose a much smaller sub-list
$\mathcal{R}_{A}\subseteq\mathcal{S}_{A}$ by picking each
$C\in\mathcal{S}_{A}$ to $\mathcal{R}_{A}$ independently at random
with probability $p=n^{t-k+\varepsilon/2}$ (we can of course assume
here and later that $\varepsilon<1\le k-t$, and thus $0<p<1$).
Finally, select $A'\in\mathcal{R}_{A}$ that has no intersection of
size at least $t$ with any $C\in\mathcal{R}_{B}$ for any other
$B\in\mathcal{L}_{\mathcal{H}}$. If there is such a choice for every
$A\in\mathcal{L}_{\mathcal{H}}$, we get $\left|A'\cap B'\right|<t$
for distinct $A,B\in\mathcal{L}_{\mathcal{H}}$, as requested.

\bigskip{}

\subsection{Details}

We start by showing that the lists $\mathcal{S}_{A}$ are large enough.
\begin{claim}
\label{clm:many-potential-candidates}For every $A\in\mathcal{L}_{\mathcal{H}}$
we have $\left|\mathcal{S}_{A}\right|=\Theta\left(n^{k-t}\right)$.\end{claim}
\begin{proof}
Fix $A\in\mathcal{L}_{\mathcal{H}}$. Obviously
$\left|\mathcal{T}_{A}\right|=\binom{n-t}{k-t}=\Theta\left(n^{k-t}\right)$.
Since $\mathcal{S}_{A}\subseteq\mathcal{T}_{A}$, it suffices to show
that
$\left|\mathcal{T}_{A}\setminus\mathcal{S}_{A}\right|=o\left(n^{k-t}\right)$.

Writing $\mathcal{L}_{H}\setminus\left\{ A\right\} $ as the disjoint
union $\bigcup_{\ell=0}^{t-1}\mathcal{B}_{\ell}$, where $\mathcal{B}_{\ell}=\left\{ B\in\mathcal{L}_{H}:\left|A\cap B\right|=\ell\right\} $,
we have
\begin{align*}
\mathcal{T}_{A}\setminus\mathcal{S}_{A} & =\left\{ C\in\mathcal{T}_{A}:\exists B\in\mathcal{L}_{H}\setminus\left\{ A\right\} \mbox{ such that }C\in\mathcal{T}_{B}\right\} \\
 & =\bigcup_{\ell=0}^{t-1}\left\{ C\in\mathcal{T}_{A}:\exists B\in\mathcal{B}_{\ell}\mbox{ such that }C\in\mathcal{T}_{B}\right\} \\
 & =\bigcup_{\ell=0}^{t-1}\bigcup_{B\in\mathcal{B}_{\ell}}\mathcal{T}_{A\cup B}.
\end{align*}
Note that for all $0\le\ell<t$ and for all $B\in\mathcal{B}_{\ell}$,
$\left|A\cup B\right|=2t-\ell$ and thus $\left|\mathcal{T}_{A\cup
B}\right|=\binom{n-2t+\ell}{k-2t+\ell}\le n^{k-2t+\ell}$. Moreover,
$\left|\mathcal{B}_{\ell}\right|=\binom{t}{\ell}\cdot
O\left(n^{t-\ell-\varepsilon}\right)=O\left(n^{t-\ell-\varepsilon}\right)$
by Property $\left(\clubsuit\right)$. Thus,
\[
\left|\mathcal{T}_{A}\setminus\mathcal{S}_{A}\right|\le\sum_{\ell=0}^{t-1}\left|\mathcal{B}_{\ell}\right|n^{k-2t+\ell}=\ell\cdot
O\left(n^{k-t-\varepsilon}\right)=o\left(n^{k-t}\right),
\]
establishing the claim.
\end{proof}
Recall that the sub-list $\mathcal{R}_{A}\subseteq\mathcal{S}_{A}$
was obtained by picking each $C\in\mathcal{S}_{A}$ to
$\mathcal{R}_{A}$ independently at random with probability
$p=n^{t-k+\varepsilon/2}$. The next claim shows that
$\mathcal{R}_{A}$ typically contains many $k$-sets whose pairwise
intersections are exactly $A$. This will be used in the proof of
Claim~\ref{clm:one-good-set}.
\begin{claim}
\label{clm:many-disjoint-candidates}Almost surely (i.e., with
probability tending to $1$ as $n$ tends to infinity), for every
$A\in\mathcal{L}_{H}$, the family $\mathcal{R}_{A}$ contains a
subset $\mathcal{Q}_{A}\subseteq\mathcal{R}_{A}$ of size
$\Theta\left(n^{\varepsilon/3}\right)$ such that $C_{1}\cap C_{2}=A$
for every two distinct $C_{1},C_{2}\in\mathcal{Q}_{A}$.\end{claim}
\begin{proof}
Fix $A\in\mathcal{L}_{H}$. Construct $\mathcal{Q}_{A}$ greedily as
follows: start with $\mathcal{Q}_{A}=\varnothing$; as long as
$\left|\mathcal{Q}_{A}\right|<n^{\varepsilon/3}$ and there exists
$C\in\mathcal{R}_{A}\setminus\mathcal{Q}_{A}$ such that $C\cap C'=A$
for all $C'\in\mathcal{Q}_{A}$, add $C$ to $\mathcal{Q}_{A}$. It
suffices to show that this process continues $\left\lfloor
n^{\varepsilon/3}\right\rfloor$ steps.

If the process halts after $s<\left\lfloor
n^{\varepsilon/3}\right\rfloor$ steps, then every $k$-tuple of
$\mathcal{R}_{A}$ intersects one of the $s$ previously chosen sets
in some vertex outside $A$. This means that there exists a subset
$X_A\subset X$  of cardinality $|X_A|=kn^{\varepsilon/3}$ ($X_A$
contains the union of these $s$ previously picked sets) such that
none of the edges $C$ of $\mathcal{S}_{A}$  satisfying $C\cap X_A=A$
is chosen into $\mathcal{R}_{A}$. For bounding the number of such
edges in $\mathcal S_A$ from bellow, we need to subtract from
$|\mathcal S_{A}|$ the number of edges $C\in \mathcal T_{A}$ with
$C\cap X_A\neq A$. The latter can be bounded (from above) by
$\sum_{i=1}^{k-t}\binom{|X_A|}{i}\binom{n-|X_A|}{k-t-i}$ (choose
$1\leq i\leq k-t$ vertices from $X_A$, other than $A$, to be in $C$,
and then choose the remaining $k-t-i$ vertices from $X\setminus
X_A$). Since
$\left|\mathcal{S}_{A}\right|=\Theta\left(n^{k-t}\right)$, and since
$$\sum_{i=1}^{k-t}\binom{|X_A|}{i}\binom{n-|X_A|}{k-t-i}\leq
\sum_{i=1}^{k-t}\Theta(n^{i \varepsilon/3})\cdot
n^{k-t-i}=o(n^{k-t}),$$ we obtain that the number of such edges in
$\mathcal{S}_{A}$ is at least
$\left|\mathcal{S}_{A}\right|-\sum_{i=1}^{k-t}\binom{|X_A|}{i}\binom{n-|X_A|}{k-t-i}=\Theta\left(n^{k-t}\right)$.
It thus follows that the probability of the latter event to happen
for a given $A$ is at most
$$
\binom{n}{kn^{\varepsilon/3}}\,(1-p)^{\Theta\left(n^{k-t}\right)}
$$
(choose $X_A$ first, and then require all edges of $\mathcal{S}_{A}$
intersecting $X_A$ only at $A$ to be absent from $\mathcal{R}_{A}$).
The above estimate is clearly at most
$$
n^{n^{\varepsilon/3}}\cdot
e^{-\Theta\left(pn^{k-t}\right)}=\exp\left\{n^{\varepsilon/3}\ln
n-\Theta\left(n^{\varepsilon/2}\right)\right\}<
\exp\left\{-n^{\varepsilon/3}\right\}\,.
$$
Taking the union bound over all ($\leq {n \choose t}$) choices of
$A$ establishes the claim.
\end{proof}

The last step is to select a well-behaved set $A'\in\mathcal{R}_{A}$.
The next claim shows this is indeed possible.
\begin{claim}
\label{clm:one-good-set}Almost surely for every $A\in\mathcal{L}_{H}$
we can select $A'\in\mathcal{R}_{A}$ such that $\left|A'\cap C\right|<t$
for all $C\in\bigcup\left\{ \mathcal{R}_{B}:B\in\mathcal{L}_{H}\setminus\left\{ A\right\} \right\} $.\end{claim}
\begin{proof}
Fix $A\in\mathcal{L}_{H}$ and fix all the random choices which
determine the list $\mathcal{R}_{A}$ such that it satisfies
Claim~\ref{clm:many-disjoint-candidates}. Let
$\mathcal{Q}_{A}\subseteq\mathcal{R}_{A}$ be as provided by
Claim~\ref{clm:many-disjoint-candidates} and let
$\mathcal{R}=\bigcup\left\{
\mathcal{R}_{B}:B\in\mathcal{L}_{H}\setminus\left\{ A\right\}
\right\} $ be the random family of all obstacle sets.  Define the
random variable $Z$ to be the number of sets $A'\in\mathcal{Q}_{A}$
for which $\left|A'\cap C\right|\ge t$ for some $C\in\mathcal{R}$.
Since $\mathcal{S}_{A}$ is disjoint from $\mathcal{S}=\bigcup\left\{
\mathcal{S}_{B}:B\in\mathcal{L}_{H}\setminus\left\{ A\right\}
\right\} $, we can view $\mathcal{R}$ as a random subset of
$\mathcal{S}$, with each element selected to $\mathcal{R}$
independently with probability $p=n^{t-k+\varepsilon/2}$. Thus $Z$
is determined by $\left|\mathcal{S}\right|$ independent trials. We
wish to show that $Z$ is not too large via
Theorem~\ref{thm:talagrand-ineq}; for this, $Z$ has to satisfy the
two conditions therein.
\begin{enumerate}
\item If $C\in\mathcal{S}$ satisfies $\left|A'\cap C\right|\ge t$ for
some $A'\in\mathcal{Q}_{A}$ then $C\setminus A$ must intersect $A'\setminus A$
(since $A\not\subseteq C$). However, the $\left(k-t\right)$-sets
$\left\{ A'\setminus A:A'\in\mathcal{Q}_{A}\right\} $ are pairwise
disjoint (by the definition of $\mathcal{Q}_{A}$) so each $C$ cannot
rule out more than $k$ different sets $A'\in\mathcal{Q}_{A}$. Thus
$Z$ is $k$-Lipschitz.
\item Assume that $Z\ge s$. Then, by definition, there exist distinct sets
$A'_{1},\ldots,A_{s}'\in\mathcal{Q}_{A}$ and (not necessarily distinct)
sets $C_{1},\ldots,C_{s}\in\mathcal{R}$ such that $\left|A_{i}'\cap C_{i}\right|\ge t$
for $i=1,\ldots,s$. These are at most $s$ trials whose outcomes
ensure that $Z\ge s$; i.e., $Z$ is certifiable.
\end{enumerate}
Let us now calculate $\E\left[Z\right]$. Fix $A'\in\mathcal{Q}_{A}$
and let $Z_{A'}$ be the indicator random variable of the event $E_{A'}=\left\{ \exists C\in\mathcal{R}:\left|A'\cap C\right|\ge t\right\} $.
The only set in $\mathcal{L}_{H}$ fully contained in $A'$ is $A$,
so we can write $\mathcal{L}_{H}\setminus\left\{ A\right\} $ as the
disjoint union $\bigcup_{\ell=0}^{t-1}\mathcal{B}_{\ell}'$, where
$\mathcal{B}_{\ell}'=\left\{ B\in\mathcal{L}_{H}:\left|A'\cap B\right|=\ell\right\} $.
For any $0\le\ell<t$ and $B\in\mathcal{B}_{\ell}'$, the number of
bad sets (i.e., sets that will trigger $E_{A'}$) in $\mathcal{S}_{B}$
is
\begin{align*}
\left|\left\{ C\in\mathcal{S}_{B}:\left|A'\cap C\right|\ge t\right\} \right| & \le\left|\left\{ C\in\mathcal{T}_{B}:\left|A'\cap C\right|\ge t\right\} \right|\\
 & =\left|\left\{ C\in\mathcal{T}_{B}:\left|\left(A'\cap C\right)\setminus B\right|\ge t-\ell\right\} \right|\\
 & =\sum_{i=t-\ell}^{k-t}\binom{k-\ell}{i}\binom{n-k-t+\ell}{k-t-i}=O\left(n^{k-2t+\ell}\right),
\end{align*}
since $C$ contains $B$, $\left|B\right|=t$, together with $i\ge
t-\ell$ elements from $A'\setminus B$ and the rest from
$X\setminus\left(A'\cup B\right)$. Each such bad set ends up in
$\mathcal{R}_{B}$ with probability $p=n^{\varepsilon/2+t-k}$, so the
expected number of bad sets in $\mathcal{R}_{B}$ is
$O\left(n^{\varepsilon/2-t+\ell}\right)$. By
Property~$\left(\clubsuit\right)$ we have
$\left|\mathcal{B}_{\ell}'\right|=O\left(n^{t-\ell-\varepsilon}\right)$
and thus the total expected number of bad sets in $\mathcal{R}$ is
$\ell\cdot
O\left(n^{\varepsilon/2-\varepsilon}\right)=O\left(n^{-\varepsilon/2}\right)$.
By Markov's inequality we have
\[
\prob\left[E_{A'}\right]=O\left(n^{-\varepsilon/2}\right).
\]
Now, $Z$ is a sum of
$\left|\mathcal{Q}_{A}\right|=\Theta\left(n^{\varepsilon/3}\right)$
 random variables $Z_{A'}$ and thus
\[
\E\left[Z\right]=\left|\mathcal{Q}_{A}\right|\cdot\E\left[Z_{A'}\right]=\left|\mathcal{Q}_{A}\right|\cdot\prob\left[E_{A'}\right]=O\left(n^{-\varepsilon/6}\right)=o\left(1\right).
\]
Applying Theorem~\ref{thm:talagrand-ineq} with
$t=\left|\mathcal{Q}_{A}\right|=\Theta\left(n^{\varepsilon/3}\right)$
and $c=k=O\left(1\right)$, we get that
\[
\prob\left[\mbox{all \ensuremath{A'\in\mathcal{R}_{A}} are ruled
out}\right]\le\prob\left[Z\ge\left|\mathcal{Q}_{A}\right|\right]<2\exp\left(-\Omega\left(n^{\varepsilon/3}\right)\right).
\]
Taking the union bound over all $\left|\mathcal{L}_{H}\right|\le\binom{n}{t}$
choices of $A\in\mathcal{L}_{H}$ establishes the claim.\end{proof}

\end{document}